\documentclass{amsart}
\usepackage[utf8]{inputenc}

\usepackage{graphics}
\usepackage{thmtools}
\usepackage{wasysym}
\usepackage[T1]{fontenc}    

\usepackage{amsthm}
\usepackage{amsbsy,amsmath,amssymb,amscd,amsfonts}
\usepackage[pagebackref=true]{hyperref}

\usepackage{graphicx,float,latexsym,color}
\usepackage[font={scriptsize,it}]{caption}
\usepackage{subcaption}

\usepackage{makecell}

\usepackage[dvipsnames]{xcolor}

\newtheorem{theorem}{Theorem}
\newtheorem*{theorem*}{Theorem}
\newtheorem{observation}{Observation}

\newtheorem{corollary}{Corollary}
\newtheorem{lemma}{Lemma}
\theoremstyle{remark}

\theoremstyle{definition}

\hypersetup{
    pdftoolbar=true,        
    pdfmenubar=true,        
    pdffitwindow=false,     
    pdfstartview={FitH},    
    colorlinks=true,       
    linkcolor=OliveGreen,          
    citecolor=blue,        
    filecolor=black,      
    urlcolor=red           
}

\usepackage{lineno}

\usepackage{comment}
\usepackage{microtype}
\usepackage{footnote}

\newcommand{\E}{\mathcal{E}}
\newcommand{\T}{\mathcal{T}}
\newcommand{\C}{\mathcal{C}}
\newcommand{\F}{\mathcal{F}}

\title{Poncelet Propellers:\\Invariant Total Blade Area}
\author{Dominique Laurain} 
\author{Daniel Jaud}
\author{Dan Reznik}
\date{January, 2021}

\begin{document}

\maketitle


\begin{abstract}
Given a triangle, a trio of circumellipses can be defined, each centered on an excenter. Over the family of Poncelet 3-periodics (triangles) in a concentric ellipse pair (axis-aligned or not), the trio resembles a rotating propeller, where each ``blade'' has variable area. Amazingly, their total area is invariant, even when the ellipse pair is not axis-aligned. We also prove a closely-related invariant involving the sum of blade-to-excircle area ratios.
\end{abstract}

\section{Introduction}
\label{sec:intro}
Invariants of Poncelet N-periodics in various ellipse pairs have been a recent focus of research \cite{reznik2020-intelligencer,garcia2020-new-properties,garcia2020-relating}. For the confocal pair alone (elliptic billiard), 80+ invariants have been catalogued \cite{reznik2020-invariants}, and several proofs have ensued  \cite{akopyan2020-invariants,bialy2020-invariants,caliz2020-area-product}.

We start by describing an invariant manifested by Poncelet 3-periodics (triangles) inscribed in an ellipse and circumscribed about a concentric circle; see Figure~\ref{fig:poncelet-incircle}. For such a pair to admit a 3-periodic family, its axes are constrained in a simple way, explained below.

\begin{figure}
    \centering
    \includegraphics[width=.7\textwidth]{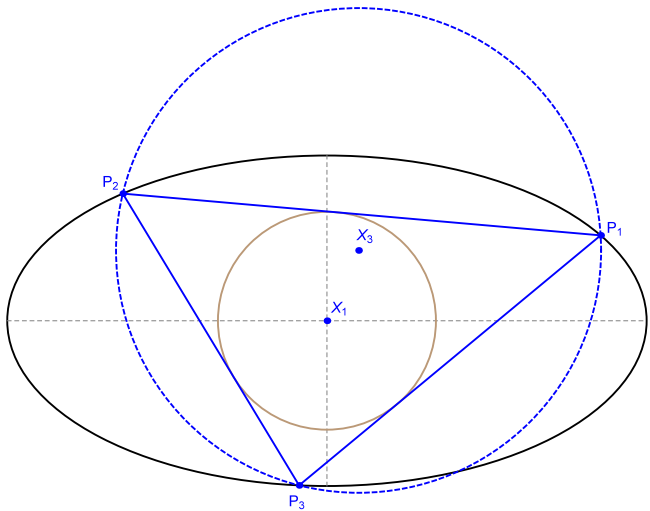}
    \caption{A Poncelet 3-periodic (blue) inscribed in an outer ellipse (black) and circumscribed about an inner concentric circle (brown). Over the family, the circumrcircle (dashed blue) has constant circumradius $R=(a+b)/2$ \cite[Thm 1]{garcia2020-relating}. \href{https://youtu.be/eIxb1so6ORo}{Video}, \href{https://bit.ly/3oF1ujm}{app}}
    \label{fig:poncelet-incircle}
\end{figure}

By definition, the incenter $X_1$ of this family is stationary at the common center and the inradius $r$ is fixed. Remarkably, the circumradius $R$ is also invariant, and this implies the sum of cosines of its internal angles is as well \cite{garcia2020-relating}.

Referring to Figure~\ref{fig:single-circum}, consider a circumellipse \cite[Circumellipse]{mw} of a 3-periodic in the incircle pair centered on one of the excenters, i.e., a vertex of the excentral triangle \cite[Excentral Triangle]{mw}. 

\begin{figure}
    \centering
    \includegraphics[width=.9\textwidth]{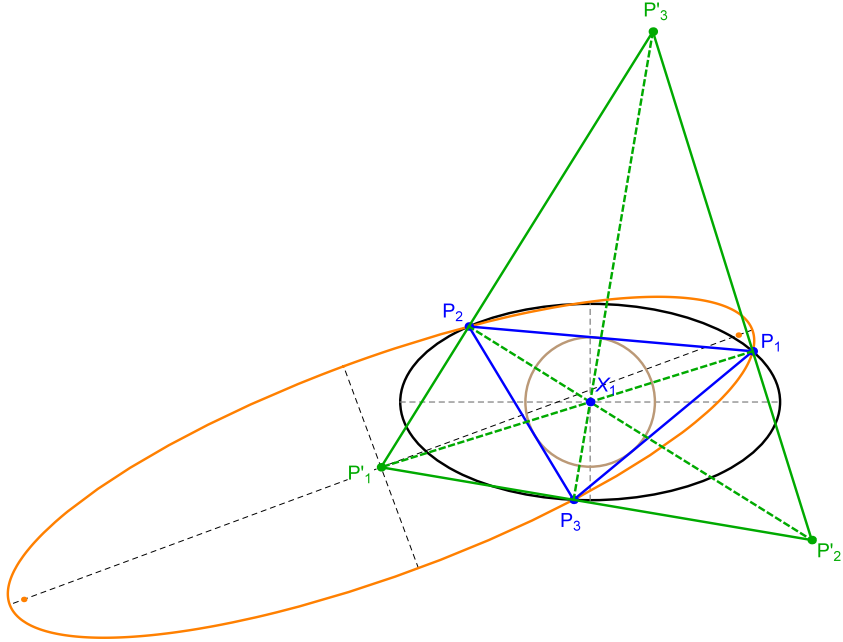}
    \caption{A Poncelet 3-periodic (blue) is shown interscribed between an outer ellipse (black) and an inner concentric circle (brown). By definition, the incenter $X_1$ is stationary at the common center. The excentral triangle (green) is the anti-cevian with respect to $X_1$. A circumellipse (orange) centered on an excenter $P_1'$ is shown. \href{https://youtu.be/JUCmAMsfdkI}{Video}}
    \label{fig:single-circum}
\end{figure}

\subsection*{Main Result} Referring to Figure~\ref{fig:three-circum}, over 3-periodics in the concentric ellipse pair with an incircle, the total area of the three excenter-centered circumellipses (blades of a ``Poncelet propeller'') is invariant.

We then extend this property to a 3-periodic family interscribed in any pair of concentric ellipses, including non-axis-aligned; see Figure~\ref{fig:non-axis-aligned}. The only proviso is that the circumellipses be centered on the vertices of the {\em anticevian triangle} with respect to the common center \cite[Anticevian Triangle]{mw}.

\begin{figure}
    \centering
    \includegraphics[width=.9\textwidth]{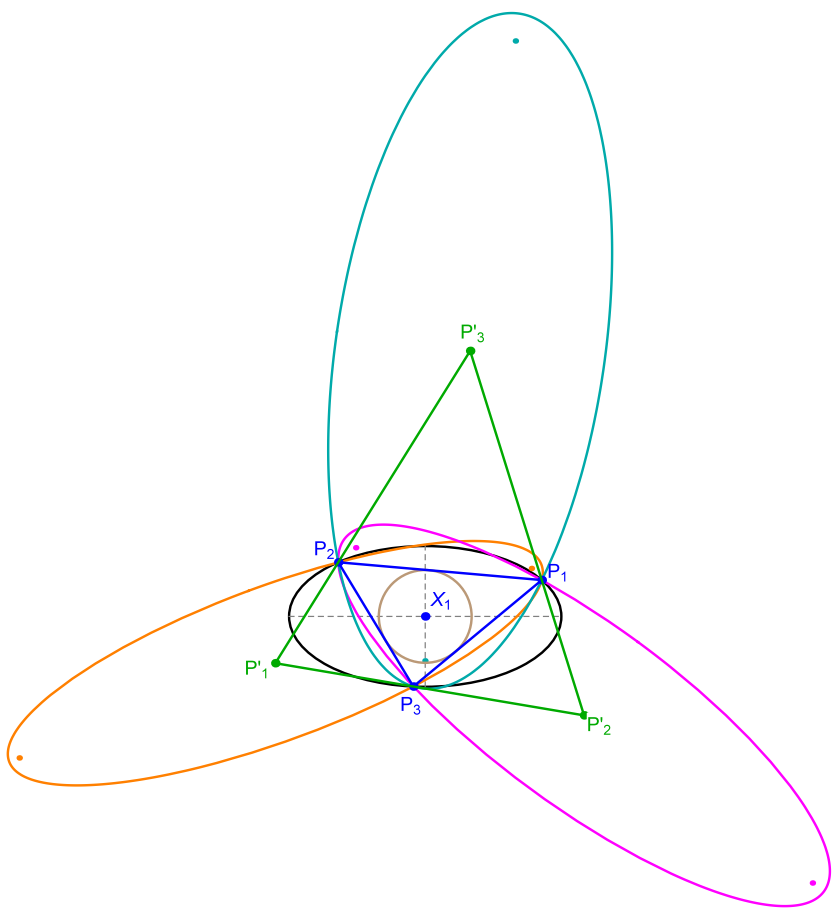}
    \caption{The three excenter-centered circumellipses (propeller blades) are shown (orange, pink, light blue) centered on the 3 excenters $P_i'$, $i=1,2,3$. Our main result is that over the 3-periodic Poncelet ``incircle'' family, the sum of their areas is invariant. \href{https://youtu.be/ub4wAv8Hgb0}{Video}}
    \label{fig:three-circum}
\end{figure}

\subsection*{Structure of Article} In Section~\ref{sec:anticevian} we derive the areas for the 3 circumellipses centered on the vertices of the anticevian triangle with respect to some point $X$. In Section~\ref{sec:incircle} we apply those formulas for 3-periodics in a pair with incircle. In Section~\ref{sec:general} we generalize the result to an unaligned concentric ellipse pair. In Section~\ref{sec:excircles} we prove a related invariant contributed by L. Gheorghe \cite{gheorghe2020-excircles} involving area ratios of circumellipses and excircles. We encourage the reader to watch some of the videos mentioned herein listed in Section~\ref{sec:videos}.

%

\section{Areas of Anticevian Circumellipses}
\label{sec:anticevian}
Let $\T=P_i$, $i=1,2,3$ be a reference triangle. Let $\C_X$ denote the circumellipse \cite[Circumellipse]{mw} centered on point $X$ interior to $\T$ and let $\Delta_X$ denote its area. Let the trilinear 
coordinates of $X$ be $[u,v,w]$ \cite[Trilinear Coordinates]{mw}. Let $\T_X$ denote the anticevian triangle of $\T$ wrt $X$, and $P_i'$ its vertices \cite[Anticevian Triangle]{mw}. Let $\C_i'$ denote the three circumellipses centered on $P_i'$. Let $\Delta_i$ denote their areas, respectively.

\begin{lemma}
The areas of the aforementioned circumellipses are given by:
\begin{align}
    \Delta_X & = \frac{z s_1 s_2 s_3 u v w}{s_1 u + s_2 v + s_3 w} \label{eq:delta-o} \\
    \Delta_1 & = \frac{z s_1 s_2 s_3 u v w }{-s_1 u + s_2 v + s_3 w} \nonumber\\
    \Delta_2 & = \frac{z s_1 s_2 s_3 u v w }{s_1 u - s_2 v + s_3 w} \nonumber\\
    \Delta_3 & = \frac{z s_1 s_2 s_3 u v w }{s_1 u + s_2 v - s_3 w} \nonumber
\end{align} 

{\small
\[  z =  \pi \sqrt{  \frac{(s_1 + s_2 + s_3) (-s_1 + s_2 + s_3) (s_1 - s_2 + s_3) (s_1 + s_2 - s_3)}{(s_1 u + s_2 v + s_3 w) (-s_1 u + s_2 v + s_3 w)(s_1 u - s_2 v + s_3 w)(s_1 u + s_2 v - s_3 w) }} \]
}

\end{lemma}

\begin{proof}
Let the trilinears of $X$ be $[u:v:w]$. The trilinears of $P_i'$ are known to be $[-u:v:w]$, $[u:-v:w]$ and $[u:v:-w]$, see \cite[Anticevian Triangle]{mw}. The expressions are obtained from the formula of the area of a circumellipse centered on $X=[u:v:w]$ \cite[Circumellipse]{mw}.
\begin{align*}
    \Delta_X = \frac{z s_1 s_2 s_3 u v w}{s_1 u + s_2 v + s_3 w}
\end{align*}

\end{proof}

The above also implies:

\begin{corollary}
For any triangle the following relation holds:
\begin{equation*}
    \frac{1}{\Delta_o} =  \frac{1}{\Delta_1} + \frac{1}{\Delta_2} + \frac{1}{\Delta_3} 
\end{equation*}
\end{corollary}

\noindent Note: this is reminiscent of the well-known property $1/r=\sum{1/r_i}$ where $r$ is the inradius and $r_i$ are the exradii of a triangle \cite[Excircles]{mw}. 

From direct calculations:

\begin{lemma}
Let $\Sigma_X$ denote the sum $\Delta_1 + \Delta_2 + \Delta_3$. This is given by:
\begin{equation}
 \Sigma_X = \frac{(s_1^2 u^2 + s_2^2 v^2 + s_3^2 w^2 - 2 s_1 s_2 u v - 2 s_1 s_3 u w - 2 s_2 s_3 v w)s_1 s_2 s_3 u v w z}{(s_1 u + s_2 v - s_3 w)(s_1 u - s_2 v + s_3 w)(s_1 u - s_2 v - s_3 w)}
\label{eq:sigma-123}
\end{equation}
\end{lemma}

\section{Poncelet Family with Incircle}
\label{sec:incircle}
Let $\F$ denote the 1d Poncelet family of 3-periodics inscribed in an outer ellipse $\E$ with semi-axes $a,b$ and circumscribed about a concentric circle $\E'$ with radius $r$. Assume $a>b>r>0$. Let $O$ denote the common center.

Recall Cayley's condition for the existence of a 3-periodic family interscribed between two concentric, axis-aligned ellipses \cite{georgiev2012-poncelet}:

\[ \frac{a_c}{a}+\frac{b_c}{b} = 1 \]

\noindent where $a_c$ and $b_c$ are the major and minor semi-axes of the inscribed ellipse corresponding to the caustic of the Poncelet family. For $\F$, $a_c = b_c = r$, then $\frac{a}{r} + \frac{b}{r} = 1$, i.e.:

\[ r = \frac{ab}{a + b}. \]

By definition, the incenter $X_1$ of $\F$ triangles is stationary at $O$. Also by definition, the inradius $r$ is fixed.

Also known is the fact that the circumradius $R$ of $\F$ triangles is invariant and given by \cite{garcia2020-relating}:

\[ R = \frac{a b}{2 r} = \frac{a + b}{2}. \]

Let $\rho=\frac{r}{R}$ denote the ratio $r/R$. From the above it is invariant and given by:

\[ \rho = \frac{r}{R} =\frac{2 r^2}{a b}. \]

Taking $X=O$, note that the area $\Delta_o=\pi a b$ of $\E$ is by definition, invariant. Also note the anticevian $\T_o$ in this case is the excentral triangle \cite[Excentral Triangle]{mw}.

\begin{lemma}
Over $\F$, $\Sigma_o$ is invariant and given by:

\[ \Sigma_o = \left(1 + \frac{4}{\rho}\right) \Delta_o\]
\label{lem:incircle}
\end{lemma}

\begin{proof}
%
   





By using well-known formulas for $r$ and $R$ \cite[Inradius,Circumradius]{mw}, one can express $\rho$ in terms of the sidelengths:

\[ 
\rho = \frac{(s_1 + s_2 - s_3)(s_1 - s_2 + s_3)(-s_1 + s_2 + s_3)}{2 s_1 s_2 s_3}
\]

Using the trilinears for the incenter $O=X_1=[u,v,w]=[1:1:1]$ in \eqref{eq:sigma-123}, direct calculations yield the claim.
\end{proof}

\section{Generalizing the Result}
\label{sec:general}
It turns out Lemma~\ref{lem:incircle} can be generalized to a larger class of Poncelet 3-periodic families.

Let $(\E,\E')$ be a generic pair of concentric ellipses (axis-aligned or not), admitting a Poncelet 3-periodic family. Let $O$ be their common center. Let $C_i'$ be the circumellipses centered on the vertices of the anticevian with respect to $O$. Referring to Figure~\ref{fig:non-axis-aligned}:

\begin{theorem}
 Over the 3-periodic family interscribed in $(\E,\E')$, $\Sigma_{o}$ is invariant.
 \label{thm:main}
\end{theorem}

\begin{proof}
$(\E,\E')$ can be regarded as an affine image of the original family $\F$ (with incircle). Let matrix $A$ represent the required affine transform. Since conics are equivariant under affine transformations\footnote{In fact, they are equivariant under projective transformations \cite{akopyan2007-conics}.} of the 5 constraints that define them (in our case, passing through the vertices of the 3-periodic and being centered on an anticevian vertex) \cite{akopyan2007-conics}, the area of each circumellipse will scale by $\det(A)$, so by Lemma~\ref{lem:incircle} the result follows.


\end{proof}

\begin{figure}
    \centering
    \includegraphics[width=.7\textwidth]{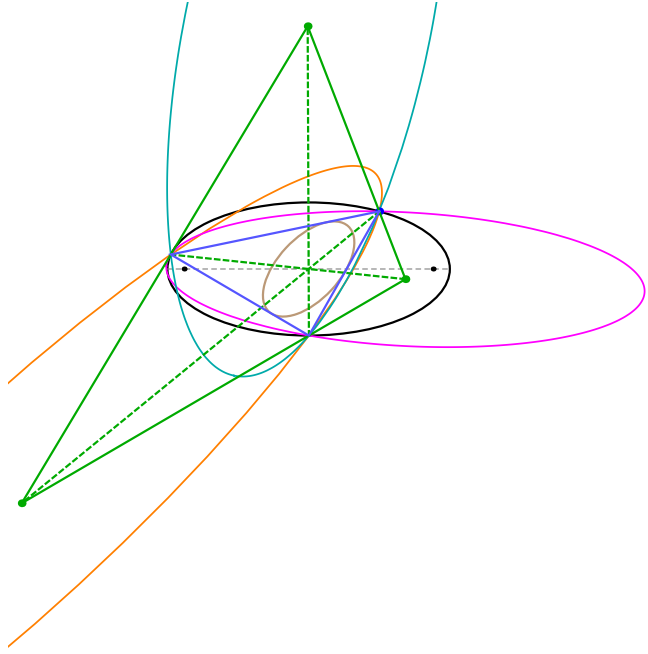}
    \caption{Consider Poncelet 3-periodics (blue) interscribed in a concentric, pair of ellipses which in general is not axis-aligned. The total area of the three circumellipses (orange, light blue, pink) centered on the vertices of the $O$-anticevian triangle (green) is invariant. \href{https://youtu.be/FJXMpUcslaA}{Video}}
    \label{fig:non-axis-aligned}
\end{figure}

In Figure~\ref{fig:anticevians}, a few other concentric, axis-aligned pairs are shown which illustrate the above corollary.

\begin{observation}
If $(\E,\E')$ are homothetic and concentric, then each of the $\Delta_i$ are constant.
\end{observation}

This arises from the fact that this pair is affinely-related to a pair of concentric circles. The associated Poncelet 3-periodic is then given by an equilateral triangle where $\Delta_1=\Delta_2=\Delta_3$.

\begin{figure}
    \centering
    \includegraphics[width=\textwidth]{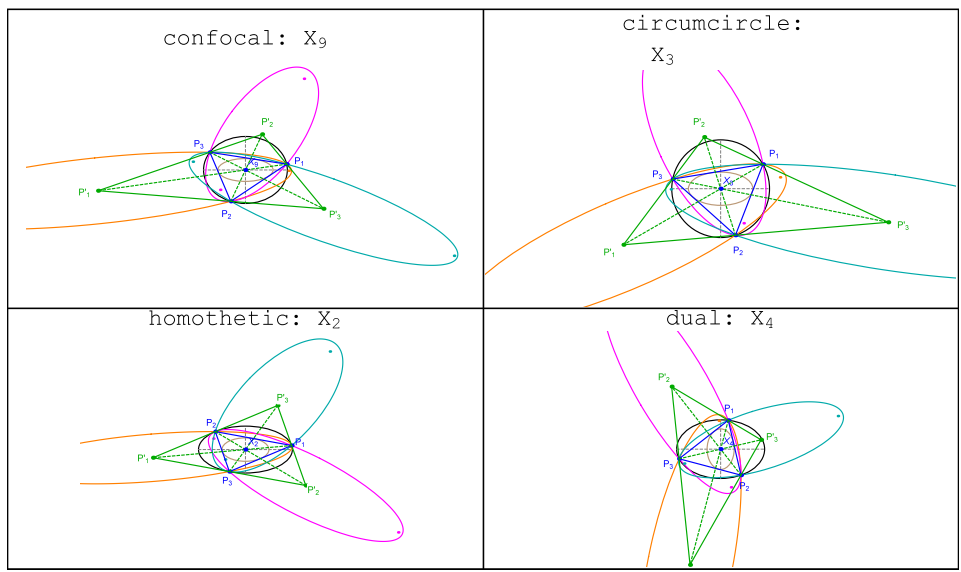}
    \caption{A picture of four 3-periodic families. The circumellipses are centered on vertices of the anticevian with respect to the center. Since each case is affinely related to $\F$, the total circumellipse area is invariant. For the homothetic pair (bottom left), the area of {\em each} circumellipse is invariant. \href{https://youtu.be/crXxPJ93ZDk}{Video}}
    \label{fig:anticevians}
\end{figure}

\section{Circumellipses Meet Excircles}
\label{sec:excircles}
Referring to Figure~\ref{fig:excircles}, L. Gheorghe detexted experimentally that the sum of area ratios of excircles to excentral circumellipses is invariant for two Poncelet families (see below) \cite{gheorghe2020-excircles}. Below we prove this and derive explicit values for the invariants. As before, let $\Delta_i$ denote the area of a circumellipse centered on the ith excenter. Let $\Omega_i$ denote the area of the ith excircle.

\begin{figure}
    \centering
    \includegraphics[width=\textwidth]{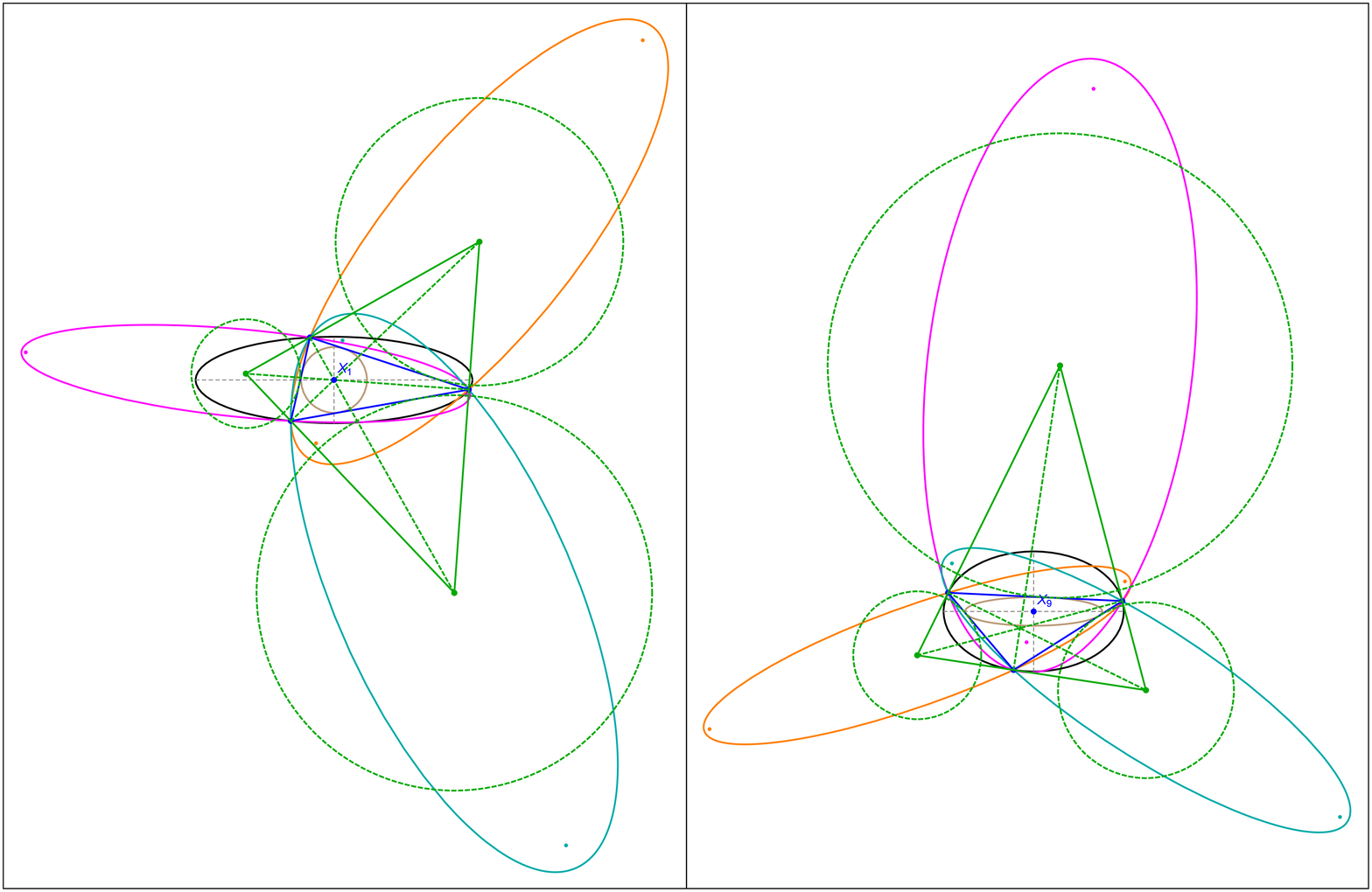}
    \caption{\textbf{Left}: The pair with incircle centered on $X_1$. Shown also are a 3-periodic (blue), the excentral triangle (solid green), the excircles (dashed green), and the three excentral circumellipses (magenta, light blue, and orange). \textbf{Right}: the same arrangement for the confocal pair, centered on $X_9$.}
    \label{fig:excircles}
\end{figure}

\begin{theorem}
Over the 3-periodic family $\F$ (with incircle):

\begin{equation*}
    \frac{\Delta_1}{\Omega_1} + \frac{\Delta_2}{\Omega_2} + \frac{\Delta_3}{\Omega_3} = \frac{2}{\rho}
\end{equation*}
\label{thm:excircles-incircle}
\end{theorem}

\begin{proof}
The exradii are given by \cite[Excircles]{mw}:

\[ r_1  = \frac{S}{s - s_1},\;\;r_2 = \frac{S}{s - s_2},\;\;r_3 = \frac{S}{s - s_3}
\]

\noindent where $s = \frac{s_1 + s_2 + s_3}{2}$ and $S$ are semi-perimeter and area of triangle $P_1 P_2 P_3$, respectively. From them obtain the excircle areas $\Omega_i$:

\begin{align*}
    \Omega_1 & = \pi \frac{s(s - s_2)(s - s_3)}{s - s_1}\\
    \Omega_2 & = \pi \frac{s(s - s_3)(s - s_1)}{s - s_2}\\
    \Omega_3 & = \pi \frac{s(s - s_1)(s - s_2)}{s - s_3}
\end{align*}

\noindent For $i=1,2,3$, the following can be derived:

\[ \frac{\Delta_i}{\Omega_i} = \mu (s - s_i) \]

\noindent with $\mu = \frac{s_1 s_2 s_3}{2s(s - s_1)(s - s_2)(s - s_3)}$. Therefore: 

\[ \sum_{i=1}^{3}{\frac{\Delta_i}{\Omega_i}}  = \mu s = \frac{2}{\rho} \]

\noindent Recall $\rho$ is constant for family $\F$ so the result for the pair incircle follows.
\end{proof}




Using an analogous proof method:

\begin{theorem}
Over the 3-periodic family interscribed in the confocal the following quantity is also invariant:

\begin{equation*}
    \frac{\Delta_1}{\Omega_1} + \frac{\Delta_2}{\Omega_2} + \frac{\Delta_3}{\Omega_3} =\frac{2}{\rho}
\end{equation*}
\label{thm:excircles-confocal}
\end{theorem}

\begin{observation}
Although the pair with incircle and the confocal pair are affinely-related, neither Theorem~\ref{thm:excircles-incircle} nor \ref{thm:excircles-confocal} for any other ellipse pair in the affine continuum.
\end{observation}

\section{List of Videos}
\label{sec:videos}
Animations illustrating some of the above phenomena are listed on Table~\ref{tab:playlist}.

\begin{table}[H]
\small
\begin{tabular}{|c|l|l|}
\hline
id & Title & \textbf{youtu.be/<.>}\\
\hline
01 & {3-Periodic incircle family has invariant $R$} &
\href{https://youtu.be/eIxb1so6ORo}{\texttt{eIxb1so6ORo}} \\
02 & {Family of incircle 3-periodics and one excentral circumellipse} &
\href{https://youtu.be/JUCmAMsfdkI}{\texttt{JUCmAMsfdkI}} \\
03 & {Three excentral circumellipses with invariant total area} &
\href{https://youtu.be/ub4wAv8Hgb0}{\texttt{ub4wAv8Hgb0}} \\
04 & {3-periodics in 5 concentric, axis-aligned ellipse pairs} &
\href{https://youtu.be/8hkeksAsx0E}{\texttt{8hkeksAsx0E}} \\
05 & {Incircle family} &
\href{https://youtu.be/tHUDfx9o0Wg}{\texttt{tHUDfx9o0Wg}} \\
06 & {Four Poncelet families} &
\href{https://youtu.be/crXxPJ93ZDk}{\texttt{crXxPJ93ZDk}} \\
07 & {Non-concentric ellipse pair} & \href{https://youtu.be/FJXMpUcslaA}{\texttt{FJXMpUcslaA}} \\
\hline
\end{tabular}
\caption{Videos of some focus-inversive phenomena. The last column is clickable and provides the YouTube code.}
\label{tab:playlist}
\end{table}

\noindent We would like to thank Liliana Gheorghe, Ronaldo Garcia, and Arseniy Akopyan for valuable discussions and contributions. 

\appendix

\section{Table of Symbols}
\label{app:symbols}
\begin{table}[H]
\begin{tabular}{|c|l|}
\hline
symbol & meaning \\
\hline

$\E,\E_c$ & outer and inner ellipses \\
$a,b$ & outer ellipse semi-axes' lengths \\
$a_c,b_c$ & inner ellipse semi-axes' lengths \\
$O$ & common center of ellipses \\
$P_1,P_2,P_3$ & 3-periodic vertices \\
$s_1,s_2,s_3$ & 3-periodic sidelengths \\
$P_1',P_2',P_3'$ & vertices of the anticevian wrt $O$ \\
$\C_X,\Delta_X$ & $X$-centered circumellipse and its area \\
$\C_i',\Delta_i$ & $P_i'$-centered circumellipse and its area \\
$\T_X$ & anticevian wrt $X$ \\
$\Sigma_{X}$ & \makecell[lt]{area sum of 3 circumellipses centered\\ on the vertices of $\T_X$} \\
$r,R$ & 3-periodic inradius and circumradius \\
$\rho$ & ratio $r/R$ \\
$[u:v:w]$ & trilinear coordinates \\
$\Omega_i$ & area of the excircle corresponding to $P_i$ \\
\hline
\end{tabular}
\caption{Symbols of euclidean geometry used}
\label{tab:symbols}
\end{table}

\bibliographystyle{maa}
\bibliography{references,authors_rgk}

\end{document}